\documentclass[12pt]{extarticle}
\usepackage{amsmath}
\usepackage{amsthm, amssymb, hyperref, color}
\usepackage[shortlabels]{enumitem}
\usepackage{graphicx}
\usepackage[all]{xypic}
\usepackage{makecell}
\usepackage[final]{pdfpages}
\usepackage{amsfonts}
\setboolean{@twoside}{false}
\usepackage{pdfpages}
\usepackage{subfigure}
\usepackage{scalefnt}
\usepackage{verbatim}
\usepackage{color}
\oddsidemargin \evensidemargin
\date{}

\usepackage{algorithm} 
\usepackage{algorithmic}

\newtheorem{theorem}{Theorem}
\numberwithin{theorem}{section}
\newtheorem{proposition}[theorem]{Proposition}

\newtheorem{corollary}[theorem]{Corollary}

\newtheorem{example}[theorem]{Example}

\DeclareMathOperator*{\tr}{tr}
\newcommand{\s}{\mathbb{S}}
\newcommand{\R}{\mathbb{R}}
\newcommand\independent{\protect\mathpalette{\protect\independenT}{\perp}}
    \def\independenT#1#2{\mathrel{\rlap{$#1#2$}\mkern2mu{#1#2}}}

\title{\textbf{Gaussian Graphical Models: \\ An Algebraic and Geometric Perspective}}
\author{Caroline Uhler}
\begin{document}

\maketitle

\begin{abstract}
\noindent 
Gaussian graphical models are used throughout the natural sciences, social sciences, and economics to model the statistical relationships between variables of interest in the form of a graph. We here provide a pedagogic introduction to Gaussian graphical models and review recent results on maximum likelihood estimation for such \mbox{models.} Throughout, we highlight the rich algebraic and geometric properties of Gaussian graphical models and explain how these properties relate to convex optimization and ultimately result in insights on the existence of the maximum likelihood estimator (MLE) and algorithms for computing the MLE. 
\end{abstract}

\section{Introduction}

Technological advances and the information era allow the collection of massive amounts of data at unprecedented resolution. Making use of this data to gain insight into complex phenomena requires characterizing the relationships among a large number of variables. Gaussian graphical models explicitly capture the statistical relationships between the variables of interest in the form of a graph. These models are used throughout the natural sciences, social sciences, and economics, in particular in computational biology, finance, and speech recognition (see e.g.~[12, 40, 100]). 

As we will see in this overview, assuming Gaussianity leads to a rich geometric structure that can be exploited for parameter estimation. However, Gaussianity is not only assumed for mathematical simplicity. As a consequence of the central limit theorem, physical quantities that are expected to be the sum of many independent contributions often follow approximately a Gaussian distribution. For example, people's height is approximately normally distributed; height is believed to be the sum of many independent contributions from various genetic and environmental factors.


Another reason for assuming normality is that the Gaussian distribution has maximum entropy among all real-valued distributions with a specified mean and covariance. Hence, assuming Gaussianity imposes the least number of structural constraints beyond the first and second moments. So another reason for assuming Gaussianity is that it is the least-informative distribution. In addition, many physical systems tend to move towards maximal entropy configurations over time. 

In the following, we denote by $\s^p$ the vector space of real symmetric $p\times p$ matrices. This vector space is equipped with the \emph{trace inner product} $\,\langle A, B \rangle  := {\rm tr}(A B)$. In addition, we denote by $\s^p_{\succeq 0}$ the convex cone of positive semidefinite matrices. Its interior is the open cone $\s^p_{\succ 0}$ of positive definite matrices. A random vector $X\in \R^p$ is distributed according to the \emph{multivariate Gaussian distribution}  $\mathcal{N}(\mu,\Sigma)$ with parameters $\mu \in \R^p$ (the \emph{mean}) and $\Sigma\in \s^p_{\succ 0}$ (the \emph{covariance matrix}), if it has density function
$$
f_{\mu,\Sigma}(x)=(2\pi)^{-p/2}(\det\Sigma)^{-1/2}\exp\left\{ -\frac{1}{2} (x-\mu)^T\Sigma^{-1}(x-\mu)\right\}, \quad x\in \R^p.
$$
In the following, we denote the inverse covariance matrix, also known as the \emph{precision matrix} or the \emph{concentration matrix}, by $K$. In terms of $K$ and using the trace inner product on $\s^p$, the density $f_{\mu,\Sigma}$ can equivalently be formulated as:
$$
f_{\mu, K}(x)=\exp\left\{ \mu^T K x - \big\langle K, \frac{1}{2} xx^T\big\rangle -\frac{p}{2}\log(2\pi) + \frac{1}{2}\log\det(K) - \frac{1}{2}\mu^TK\mu\right\}.
$$
Hence, the Gaussian distribution is an \emph{exponential family} with \emph{canonical parameters} $(-\mu^TK, K)$, \emph{sufficient statistics} $(x,\frac{1}{2}xx^T)$ and \emph{log-partition function} (also known as the \emph{cumulant generating function}) $\frac{p}{2}\log(2\pi) - \frac{1}{2}\log\det(K) + \frac{1}{2}\mu^TK\mu$; see~\cite{Barndorff_1978, Brown_exponential} for an introduction to exponential families.



Let $G = (V,E)$ be an undirected graph with vertices $V = [p]$ and edges $E$, where  $[p] = \{1,\ldots,p\}$. A random vector $X \in \mathbb{R}^p$ is said to \emph{satisfy the (undirected) Gaussian graphical model with graph $G$}, if $X$ has a multivariate Gaussian distribution $\mathcal{N}(\mu,\Sigma)$ with 
$$\big(\Sigma^{-1}\big)_{i,j} = 0 \quad \textrm{for all } (i,j)\notin E.$$
Hence, the graph $G$ describes the sparsity pattern of the concentration matrix. This explains why $G$ is also known as the \emph{concentration graph}. As we will see in Section~\ref{sec:CI}, missing edges in $G$ also correspond to conditional independence relations in the corresponding Gaussian graphical model. Hence, sparser graphs correspond to simpler models with fewer canonical parameters and more conditional independence relations.

Gaussian graphical models are the continuous counter-piece to Ising models. Like Ising models, Gaussian graphical models are quadratic exponential families. These families only model the pairwise interactions between nodes, i.e., interactions are only on the edges of the underlying graph $G$. But nevertheless, Ising models and Gaussian graphical models are extremely flexible models; in fact, they can capture any pairwise correlation structure that can be constructed for binary or for continuous data. 

This overview discusses maximum likelihood (ML) estimation for Gaussian graphical models. There are two problems of interest in this regard: (1) to estimate the edge weights, i.e.~the canonical parameters, given the graph structure, and (2) to learn the underlying graph structure. This overview is mainly focussed with the first problem (Sections~\ref{sec:likelihood}-\ref{sec:MLE_alg}), while the second problem is only discussed in Section~\ref{sec:graph_learning}. The second problem is particularly important in the high-dimensional setting when the number of samples $n$ is smaller than the number of variables $p$. For an introduction to high-dimensional statistics see e.g.~\cite{Buehlmann_vandeGeer}.

The remainder of this overview is structured as follows: In Section~\ref{sec:CI}, we examine conditional independence relations for Gaussian distributions. Then, in Section~\ref{sec:likelihood}, we introduce the Gaussian likelihood. 
We show that ML estimation for Gaussian graphical models is a convex optimization problem and we describe its dual optimization problem. In Section~\ref{sec:pd_completion}, we analyze this dual optimization problem and explain the close links to positive definite matrix completion problems studied in linear algebra. 
In Section~\ref{sec:geometry}, we develop a geometric picture of ML estimation for Gaussian graphical models that complements the point of view of convex optimization. 
The combination of convex optimization, positive definite matrix completion, and convex geometry allows us to obtain results about the existence of the maximum likelihood estimator (MLE) and algorithms for computing the MLE. These are presented in Section~\ref{sec:MLE_existence} and in Section~\ref{sec:MLE_alg}, respectively. Gaussian graphical models are defined by zero constraints on the concentration matrix $K$. In Section~\ref{sec:graph_learning}, we describe \mbox{methods} for learning the underlying graph, or equivalently, the zero pattern of~$K$. Finally, in Section~\ref{sec:linear_Gaussian}, we end with a discussion of other Gaussian models with linear constraints on the concentration matrix or the covariance matrix.

\section{Gaussian distribution and conditional independence}
\label{sec:CI}

We start this section by reviewing some of the extraordinary properties of Gaussian distributions. The following result shows that the Gaussian distribution is closed under \mbox{marginalization} and conditioning. We here only provide proofs that will be useful in later sections of this overview. A complete proof of the following well-known result can be found for example in \cite{Anderson_2003, Eaton_1983}. 

\begin{proposition}
\label{prop_Gaussian}
Let $X\in\mathbb{R}^p$ be distributed as $\mathcal{N}(\mu, \Sigma)$ and partition the random vector $X$ into two components $X_A\in\mathbb{R}^a$ and $X_B\in\mathbb{R}^b$ such that $a+b=p$. Let $\mu$ and $\Sigma$ be partitioned accordingly, i.e.,
$$\mu = \begin{pmatrix} \mu_A \\ \mu_B\end{pmatrix} \quad \textrm{and} \quad \Sigma = \begin{pmatrix} \Sigma_{A,A} & \Sigma_{A,B} \\ \Sigma_{B,A} & \Sigma_{B,B}\end{pmatrix},$$
where, for example, $\Sigma_{B,B}\in\s^b_{\succ 0}$. Then,
\begin{enumerate}
\item[(a)] the marginal distribution of $X_A$ is $\;\mathcal{N}(\mu_A, \Sigma_{A,A})$;
\item[(b)] the conditional distribution of $X_A\mid X_B = x_B$ is $\;\mathcal{N}(\mu_{A\mid B}, \Sigma_{A\mid B})$, where
$$\mu_{A\mid B} = \mu_A + \Sigma_{A,B} \Sigma_{B,B}^{-1}(x_B-\mu_B) \quad \textrm{and} \quad \Sigma_{A\mid B} = \Sigma_{A,A} - \Sigma_{A,B}\Sigma_{B,B}^{-1}\Sigma_{B,A}.$$
\end{enumerate}
\end{proposition}

\begin{proof}
We only prove (b) to demonstrate the importance of Schur complements when working with Gaussian distributions. Fixing $x_B$, we find by direct calculation that the conditional density $f(x_A\mid x_B)$ is proportional to:
\begin{align}
f(x_A\mid x_B) &\propto \exp\Big\{-\frac{1}{2}(x_A-\mu_A)^TK_{A,A}(x_A-\mu_A) -  (x_A-\mu_A)^TK_{A,B}(x_B-\mu_B)\Big\}\quad\quad\nonumber\\
&\propto \exp\Big\{-\frac{1}{2}\big(x_A-\mu_A-K_{A,A}^{-1}K_{A,B}(x_B-\mu_B)\big)^TK_{A,A}\label{eq_1}\\
&\hspace{4.8cm}\times\big(x_A-\mu_A-K_{A,A}^{-1}K_{A,B}(x_B-\mu_B)\big)\Big\},\quad\quad\nonumber
\end{align}
where we used the same partitioning for $K$ as for $\Sigma$. Using Schur complements, we obtain
$$K_{A,A}^{-1} \;=\; \Sigma_{A,A}-\Sigma_{A,B}\Sigma_{B,B}^{-1}\Sigma_{B,B},$$
and hence $K_{A,A} = \Sigma_{A\mid B}^{-1}$. Similarly, we obtain $\,K_{A,A}^{-1}K_{A,B} = -\Sigma_{A,B}\Sigma_{B,B}^{-1}$. Combining these two identities with the conditional density in (\ref{eq_1}) completes the proof.
\end{proof}

These basic properties of the multivariate Gaussian distribution have interesting implications with respect to the interpretation of zeros in the covariance and the concentration matrix. Namely, as described in the following corollary, zeros correspond to (\emph{conditional}) \emph{independence relations}. For disjoint subsets $A,B,C\subset [p]$ we denote the statement that $X_A$ is conditionally independent of $X_B$ given $X_C$ by $X_A\independent X_B\mid X_C$. If $C=\emptyset$, then we write $X_A\independent X_B$.

\begin{corollary}
\label{cor_Gaussian_CI}
Let $X\in\mathbb{R}^p$ be distributed as $\mathcal{N}(\mu, \Sigma)$ and let $i,j\in [p]$ with $i\neq j$. Then
\begin{enumerate}
\item[(a)] $X_i\independent X_j\;$ if and only if $\;\Sigma_{i,j}=0$;
\item[(b)] $X_i\independent X_j\mid X_{[p]\setminus\{i,j\}}\;$ if and only if $\;K_{i,j}=0\;$ if and only if $\;\det(\Sigma_{[p]\setminus\{i\}, [p]\setminus\{j\}})=0$.
\end{enumerate}
\end{corollary}

\begin{proof}
Statement (a) follows directly from the expression for the conditional mean in Proposition~\ref{prop_Gaussian} (b). From the expression for the conditional covariance in Proposition~\ref{prop_Gaussian} (b) it follows that $\Sigma_{\{i,j\}\mid ([p]\setminus\{i,j\})} = (K_{\{i,j\}, \{i,j\}})^{-1}$. To prove (b), note that if follows from (a) that $X_i\independent X_j\mid X_{[p]\setminus\{i,j\}}$ if and only if the $2\times 2$ conditional covariance matrix $\Sigma_{\{i,j\}\mid ([p]\setminus\{i,j\})}$ is diagonal. This is the case if and only if $K_{\{i,j\}, \{i,j\}}$ is diagonal, or equivalently, $K_{i,j}=0$. This proves the first equivalence in~(b). The second equivalence is a consequence of the cofactor formula for matrix inversion, since 
$$K_{i,j}=(\Sigma^{-1})_{i,j} = (-1)^{i+j}\frac{\det(\Sigma_{[p]\setminus\{i\}, [p]\setminus\{j\}})}{\det(\Sigma)},$$
which completes the proof.
\end{proof}

Corollary~\ref{cor_Gaussian_CI} shows that for undirected Gaussian graphical models a missing edge $(i,j)$ in the underlying graph $G$ (i.e.~the concentration graph) corresponds to the conditional independence relation $X_i\independent X_j \mid X_{[p]\setminus\{i,j\}}$. 
Corollary~\ref{cor_Gaussian_CI} can be generalized to an equivalence between any conditional independence relation and the vanishing of a particular almost principal minor of $\Sigma$ or $K$. This is shown in the following proposition.

\begin{proposition}
Let $X\in\mathbb{R}^p$ be distributed as $\,\mathcal{N}(\mu, \Sigma)$. Let $\,i,j\in [p]$ with $i\neq j$ and let $\,S\subseteq [p]\setminus\{i,j\}$. Then the following statements are equivalent:
\begin{enumerate}
\item[(a)] $X_{i} \independent X_j \mid X_S$;
\item[(b)] $\det(\Sigma_{iS, jS})=0$, where $iS=\{i\}\cup S$;
\item[(c)] $\det(K_{iR, jR})=0$, where $R = [p]\setminus(S\cup\{i,j\})$.
\end{enumerate}
\end{proposition}

\begin{proof}
By Proposition~\ref{prop_Gaussian} (a), the marginal distribution of $\,X_{S\cup\{i,j\}}$ is Gaussian with covariance matrix $\,\Sigma_{ijS, ijS}$. Then Corollary~\ref{cor_Gaussian_CI} (b) implies the equivalence between (a) and (b). Next we show the equivalence between (a) and (c): It follows from Proposition~\ref{prop_Gaussian} (b) that the inverse of $\,K_{ijR, ijR}$ is equal to the conditional covariance $\,\Sigma_{ijR\mid S}$. Hence by Corollary~\ref{cor_Gaussian_CI}~(a), the conditional independence statement in (a) is equivalent to $((K_{ijR, ijR})^{-1})_{ij}=0$, which by the cofactor formula for matrix inversion is equivalent to~(c). 
\end{proof}

\section{Gaussian likelihood and convex optimization}
\label{sec:likelihood}

Given $n$ i.i.d.~observations $X^{(1)},\dots ,X^{(n)}$ from $\,\mathcal{N}(\mu,\Sigma)$, we define the \emph{sample covariance matrix} as
\vspace{-0.2cm}
$$S\;=\;\frac{1}{n}\sum_{i=1}^n (X^{(i)}-\bar X) (X^{(i)}-\bar X)^T,\vspace{-0.2cm}$$ where $\bar X=\frac{1}{n}\sum_{i=1}^{n}X^{(i)}$ is the \emph{sample mean}. We will see that $\bar X$ and $S$ are sufficient statistics for the Gaussian model and hence we can write the log-likelihood function in terms of these quantities. Ignoring the normalizing constant, the Gaussian log-likelihood expressed as a function of $(\mu, \Sigma)$ is
\begin{align*}
\ell(\mu, \Sigma) &\propto -\frac{n}{2}\log\det (\Sigma) -\frac{1}{2}\sum_{i=1}^n (X^{(i)}-\mu)^T\Sigma^{-1}(X^{(i)}-\mu) \\
&= -\frac{n}{2}\log\det (\Sigma) - \frac{1}{2}\tr\bigg( \Sigma^{-1}\Big(\sum_{i=1}^n (X^{(i)}-\mu)(X^{(i)}-\mu)^T\Big)\bigg) \\
&= -\frac{n}{2}\log\det (\Sigma) - \frac{n}{2}\tr(S\Sigma^{-1}) -\frac{n}{2} (\bar X-\mu)^T \Sigma^{-1} (\bar X-\mu),
\end{align*}
where for the last equality we expanded $X^{(i)}-\mu = (X^{(i)}-\bar X) + (\bar X-\mu)$ and used the fact that $\sum_{i=1}^n (X^{(i)}-\bar X)=0$. Hence, it can easily be seen that in the \emph{saturated} (unconstrained) \emph{model} where $(\mu,\Sigma)\in \R^p\times\s^p_{\succ 0}$, the MLE is given by
$$\hat{\mu} = \bar X \quad \textrm{and} \quad \hat{\Sigma} = S,$$
assuming that $S\in\s^p_{\succ 0}$.

ML estimation under general constraints on the parameters $(\mu, \Sigma)$ can be complicated. Since Gaussian graphical models only pose constraints on the covariance matrix, we will restrict ourselves to models where the mean $\mu$ is unconstrained, i.e.~$(\mu, \Sigma)\in\R^p\times\Theta$, where $\Theta\subseteq\s^p_{\succ 0}$. In this case, $\hat\mu=\bar X$ and the ML estimation problem for $\Sigma$ boils down to the optimization problem
\begin{equation}
\label{opt_Sigma}
\begin{aligned}
& \underset{\Sigma}{\text{maximize}}
& &-\log\det (\Sigma) - \tr(S\Sigma^{-1}) \\
& \text{subject to}
& & \;\Sigma \in \Theta.
\end{aligned}
\end{equation}
While this objective function as a function of the covariance matrix $\Sigma$ is in general not concave over the whole cone $\s^p_{\succ 0}$, it is easy to show that it is concave over a large region of the cone, namely for all $\Sigma\in\s^p_{\succ 0}$ such that $\Sigma-2S \in\s^p_{\succ 0}$ (see~\cite[Excercise 7.4]{boyd_Vandenberghe}).

Gaussian graphical models are given by linear constraints on $K$. So it is convenient to write the optimization problem (\ref{opt_Sigma}) in terms of the concentration matrix $K$:
\begin{equation}
\label{opt_K}
\begin{aligned}
& \underset{K}{\text{maximize}}
& &\log\det (K) - \tr(S K) \\
& \text{subject to}
& & \;K \in \mathcal{K},
\end{aligned}
\end{equation}
where $\mathcal{K} = \Theta^{-1}$. In particular, for a Gaussian graphical model with graph $G=(V,E)$ the constraints are given by $K\in\mathcal{K}_G$, where
$$\mathcal{K}_G := \{K\in\s^p_{\succ 0} \mid K_{i,j}=0 \textrm{ for all } i\neq j \textrm{ with } (i,j)\notin E\}.$$
In the following, we show that the objective function in (\ref{opt_K}), i.e.~as a function of $K$, is concave over its full domain $\s^p_{\succ 0}$. Since $\mathcal{K}_G$ is a convex cone, this implies that ML estimation for Gaussian graphical models is a convex optimization problem.

\begin{proposition}
\label{prop_convex}
The function $f(Y) = \log\det (Y) - \tr(SY)$ is concave on its domain $\s^p_{\succ 0}$. 
\end{proposition}

\begin{proof}
Since $\tr(SY)$ is linear in $Y$ it suffices to prove that the function $\,\log\det(Y)$ is concave over $\s^p_{\succ 0}$. We prove this by showing that the function is concave on any line in $\s^p_{\succ 0}$. Let $Y\in\s^p_{\succ 0}$ and consider the line $Y+tV$, $V\in\s^p$, that passes through $Y$. It suffices to prove that $g(t) = \log\det (Y+tV)$ is concave for all $t\in\R$ such that $Y+tV\in\s^p_{\succ 0}$. This can be seen from the following calculation:
\begin{align*}
g(t) &= \log\det (Y+tV)\\
&= \log\det(Y^{1/2}(I+tY^{-1/2}VY^{-1/2})Y^{1/2})\\
&=\log \det(Y) + \sum_{i=1}^p\log(1+t\lambda_i),
\end{align*}
where $I$ denotes the identity matrix and $\lambda_i$ are the eigenvalues of \,$Y^{-1/2}VY^{-1/2}$. This completes the proof, since $\log\det (Y)$ is a constant and $\log(1+t\lambda_i)$ is concave in $t$.
\end{proof}

As a consequence of Proposition~\ref{prop_convex}, we can study the dual of (\ref{opt_K}) with $\mathcal{K}=\mathcal{K}_G$. See e.g.~\cite{boyd_Vandenberghe} for an introduction to convex optimization and duality theory. The Lagrangian of this convex optimization problem is given by:
\begin{align*}
\mathcal{L}(K,\nu) &= \log\det (K) - \tr(S K) -2\sum_{(i,j)\notin E, i\neq j} \nu_{i,j}K_{i,j}\\
&= \log\det (K) - \sum_{i=1}^p S_{i,i} K_{i,i} - 2\sum_{(i,j)\in E}S_{i,j} K_{i,j}- 2\sum_{(i,j)\notin E, \, i\neq j} \nu_{i,j}K_{i,j},
\end{align*}
where $\nu=(\nu_{i,j})_{(i,j)\notin E}$ are the Lagrangian multipliers. To simplify the calculations, we omit the constraint $K\in\s^p_{\succ 0}$. This can be done, since it is assumed that $K$ is in the domain of $\mathcal{L}$. Maximizing $\mathcal{L}(K,\nu)$ with respect to $K$ gives
$$(\hat{K}^{-1})_{i,j} = \left\{ \begin{array}{ll}
S_{i,j} & \textrm{if $\;i=j$ or $(i,j)\in E$}\\
\nu_{i,j} & \textrm{otherwise}.
  \end{array} \right.
$$
The Lagrange dual function is obtained by plugging in $\hat{K}$ for $K$ in $\mathcal{L}(K,\nu)$, which results in
$$g(\nu) = \log\det (\hat{K}) - \tr (\hat{K}^{-1}\hat{K}) = \log\det (\hat{K}) - p.$$
Hence, the dual optimization problem to ML estimation in Gaussian graphical models is given by
\begin{equation}
\label{opt_dual}
\begin{aligned}
& \underset{\Sigma\in\s^p_{\succ 0}}{\text{minimize}}
& &-\log\det \Sigma - p \\
& \text{subject to}
& & \Sigma_{i,j} = S_{i,j} \;\textrm{ for all } \;i=j \textrm{ or } (i,j)\in E.
\end{aligned}
\end{equation}
Note that this optimization problem corresponds to \emph{entropy maximization} for fixed sufficient statistics. In fact, this dual relationship between likelihood maximization and entropy maximization holds more generally for exponential families; see~\cite{Wainwright_Jordan}.

Sections~\ref{sec:geometry} and \ref{sec:MLE_existence} are centered around the existence of the MLE. We say that the MLE does not exist if the likelihood does not attain the global maximum. Note that the identity matrix is a strictly feasible point for (\ref{opt_K}) with $\;\mathcal{K}=\mathcal{K}_G$. Hence, the MLE does not exist if and only if the likelihood is unbounded. \emph{Slater's constraint qualification} states that the existence of a strictly primal feasible point is sufficient for strong duality to hold for a convex optimization problem (see e.g.~\cite{boyd_Vandenberghe} for an introduction to convex optimization). Since the identity matrix is a strictly feasible point for (\ref{opt_K}), strong duality holds for the optimization problems (\ref{opt_K}) with $\mathcal{K}=\mathcal{K}_G$ and (\ref{opt_dual}), and thus we can equivalently study the dual problem (\ref{opt_dual}) to obtain insight into ML estimation for Gaussian graphical models. In particular, the MLE does not exist if and only if there exists no feasible point for the dual optimization problem~(\ref{opt_dual}). In the next section, we give an algebraic description of this property. A generalization of this characterization for the existence of the MLE holds also more generally for regular exponential families; see~\cite{Barndorff_1978, Brown_exponential}.

\section{The MLE as a positive definite completion problem}
\label{sec:pd_completion}

To simplify notation, we use $E^*= E\cup\{(i,i)\mid i\in V\}$. We introduce the projection on the augmented edge set $E^*$, namely
$$\pi_{G}:\s^p_{\succeq 0} \to \R^{|E^*|}, \quad \pi_G(S) = \{S_{i,j} \mid (i,j)\in E^*\}.$$
Note that $\pi_G(S)$ can be seen as a partial matrix, where the entries corresponding to missing edges in the graph $G$ have been removed (or replaced by question marks as shown in~(\ref{ex_partial_matrix}) for the case where $G$ is the 4-cycle). In the following, we use $S_G$ to denote the partial matrix corresponding to $\pi_G(S)$. Using this notation, the constraints in the  optimization problem (\ref{opt_dual}) become $\,\Sigma_G = S_G$. Hence, existence of the MLE in a Gaussian graphical model is a \emph{positive definite matrix completion problem}: The MLE exists if and only if the partial matrix $S_G$ can be completed to a positive definite matrix. In that case, the MLE $\hat \Sigma$ is the unique positive definite completion that maximizes the determinant. And as a consequence of strong duality, we obtain that $(\hat \Sigma^{-1})_{i,j}=0$ for all $(i,j)\notin E^*$.

Positive definite completion problems have been widely studied in the linear algebra literature~\cite{Barrett2, Barrett1, Grone_1984, Laurent_1997}. Clearly, if a partial matrix has a positive definite completion, then every specified (i.e., with given entries) principal submatrix is positive definite. Hence, having a positive definite completion imposes some obvious necessary conditions. However, these conditions are in general not sufficient as seen in the following example, where the graph $G$ is the 4-cycle:
\begin{equation}
\label{ex_partial_matrix}
S_G=\begin{pmatrix} 1 & 0.9 & ? & -0.9 \\ 0.9 & 1 & 0.9 & ? \\ ? & 0.9 & 1 & 0.9 \\ -0.9 & ? & 0.9 & 1\end{pmatrix}.
\end{equation}
It can easily be checked that this partial matrix does not have a positive definite completion, although all the specified $2\times 2$-minors are positive. Hence, the MLE does not exist for the sufficient statistics given by $S_G$. 

This example leads to the question if there are graphs for which the obvious necessary conditions 
are also sufficient for the existence of a positive definite matrix completion. The following remarkable theorem proven in \cite{Grone_1984} answers this question.
\begin{theorem}
\label{thm_chordal}
For a graph $G$ the following statements are equivalent:
\begin{enumerate}
\item [(a)] A $G$-partial matrix $M_G\in\mathbb{R}^{|E^*|}$ has a positive definite completion if and only if all completely specified submatrices  in $M_G$ are positive definite.
\item [(b)] $G$ is \emph{chordal} (also known as \emph{triangulated}), i.e.~every cycle of length 4 or larger has a chord. 
\end{enumerate}\end{theorem}
The proof in~\cite{Grone_1984} is constructive. It makes use of the fact that any chordal graph can be turned into a complete graph by adding one edge at a time in such a way, that the resulting graph remains chordal at each step. Following this ordering of edge additions, the partial matrix is completed entry by entry in such a way as to maximize the determinant of the largest complete submatrix that contains the missing entry. Hence the proof in~\cite{Grone_1984} can be turned into an algorithm for finding a positive definite completion for partial matrices on chordal graphs.

We will see in Section~\ref{sec:MLE_existence} how to make use of positive definite completion results to determine the minimal number of observations required for existence of the MLE in a Gaussian graphical model.

\section{ML estimation and convex geometry}
\label{sec:geometry}


After having introduced the connections to positive definite matrix completion problems, we now discuss how convex geometry enters the picture for ML estimation in Gaussian graphical models. We already introduced the set
$$\mathcal{K}_G : = \{ K \in \mathbb{S}^p_{\succ 0} \mid K_{i,j} = 0 \;\textrm{ for all } (i, j)\notin E^* \}.$$
Note that $\mathcal{K}_G$ is a convex cone obtained by intersecting the convex cone $\s^p_{\succ 0}$ with a linear subspace. We call $\mathcal{K}_G$ the \emph{cone of concentration matrices}. 

A second convex cone that plays an important role for ML estimation in Gaussian graphical models is the \emph{cone of sufficient statistics} denoted by $\mathcal{S}_G$. It is defined as the projection of the positive semidefinite cone onto the entries $E^*$, i.e.,
$$\mathcal{S}_G:=\pi_G(\s^p_{\succeq 0}).$$
In the following proposition, we show how these two cones are related to each other.

\begin{proposition} \label{prop:duality}
Let $G$ be an undirected graph. Then the cone of sufficient statistics $\mathcal{S}_G$ is the dual
cone to the cone of concentration matrices $\mathcal{K}_G$, i.e.
\begin{equation}
\label{duality_without_basis}
 \mathcal{S}_G \,=\,\bigl\{ \,S_G\in \R^{|E^*|}  \,\mid\, \langle S_G, K \rangle \geq 0
\,\,\,\hbox{for all} \,\,\,K \in \mathcal{K}_G \bigr\}.
\end{equation}
\end{proposition}

\begin{proof}
Let $\mathcal{K}_G^\vee$ denote the dual of $\mathcal{K}_G$, i.e.~the right-hand side of (\ref{duality_without_basis}). Let 
$$\mathcal{L}_G:=\{K\in\s^p\mid K_{i,j}=0 \textrm{ for all } (i,j)\notin E^*\}$$
denote the linear subspace defined by the graph $G$. We denote by $\mathcal{L}_G^\perp$ the orthogonal complement of $\mathcal{L}_G$ in $\s^p$. 
Using the fact that the dual of the full-dimensional
cone $\s^p_{\succ 0}$ is $\s^p_{\succeq 0}$, i.e.~$(\s^p_{\succ 0})^\vee = \s^p_{\succeq 0}$, general duality theory for convex cones (see e.g.~\cite{Barvinok}) implies:
$$
\mathcal{K}_G^\vee \,\,\, = \,\,\,
( \s^p_{\succ 0} \,\cap \, \mathcal{L}_G )^\vee \,\,\, = \,\,\,
(\s^p_{\succeq 0}\,+\, \mathcal{L}_G^\perp)/\mathcal{L}_G^\perp \,\,\,\, = \,\,\,\,\mathcal{S}_G,$$
which completes the proof. 
\end{proof}

It is clear from this proof that the geometric picture we have started to draw holds more generally for any Gaussian model that is given by linear constraints on the concentration matrix. We will therefore use $\mathcal{L}$ to denote any linear subspace of $\s^p$ and we assume that $\mathcal{L}$ intersects the interior of $\s^p_{\succeq 0}$. Hence, $\mathcal{L}_G$ is a special case defined by zero constraints given by missing edges in the graph $G$. Then,
$$\mathcal{K}_{\mathcal{L}} = \mathcal{L}\cap\s^p_{\succ 0}, \quad \mathcal{S}_{\mathcal{L}} =\pi_{\mathcal{L}}(\s^p_{\succeq 0}) = \mathcal{K}_{\mathcal{L}}^\vee,$$
where $\pi_{\mathcal{L}}:\s^p\to \s^p/\mathcal{L}^\perp$. Note that given a basis \,$K_1,\dots , K_d$\, for $\mathcal{L}$, this map can be identified  with
\begin{equation*}
\pi_\mathcal{L} \,: \,\,\s^p \rightarrow \R^d, \quad S \mapsto \bigl( \langle S,K_1 \rangle,
\ldots,\langle S, K_d \rangle \bigr).
\end{equation*}

A \emph{spectrahedron} is a convex set that is defined by linear matrix inequalities. Given a sample covariance matrix $S$, we define the  spectrahedron
$$ {\rm fiber}_\mathcal{L}(S) \quad = \quad
\bigl\{ \Sigma \in \s^p_{\succ 0} \,\mid \,
\langle \Sigma, K \rangle \,  = \,
\langle S, K \rangle \,\,
\hbox{for all} \,\, K \in \mathcal{L} \bigr\}. $$
For a Gaussian graphical model with underlying graph $G$ this spectrahedron consists of all positive definite completions of $S_G$, i.e.
$$ {\rm fiber}_{G}(S) \quad = \quad
\left\{ \Sigma \in \s^p_{\succ 0} \,\mid \,
 \Sigma_G \,  = \,
S_G \right\}. $$

The following theorem combines the point of view of convex optimization developed in Section~\ref{sec:likelihood}, the connection to positive definite matrix completion discussed in Section~\ref{sec:pd_completion}, and the link to convex geometry described in this section into a result about the existence of the MLE in Gaussian models with linear constraints on the concentration matrix, which includes Gaussian graphical models as a special case. This result is essentially also given in~\cite[Theorem 2]{Grone_1984}.

\begin{figure}[b!]
\centering
\includegraphics[scale=0.53]{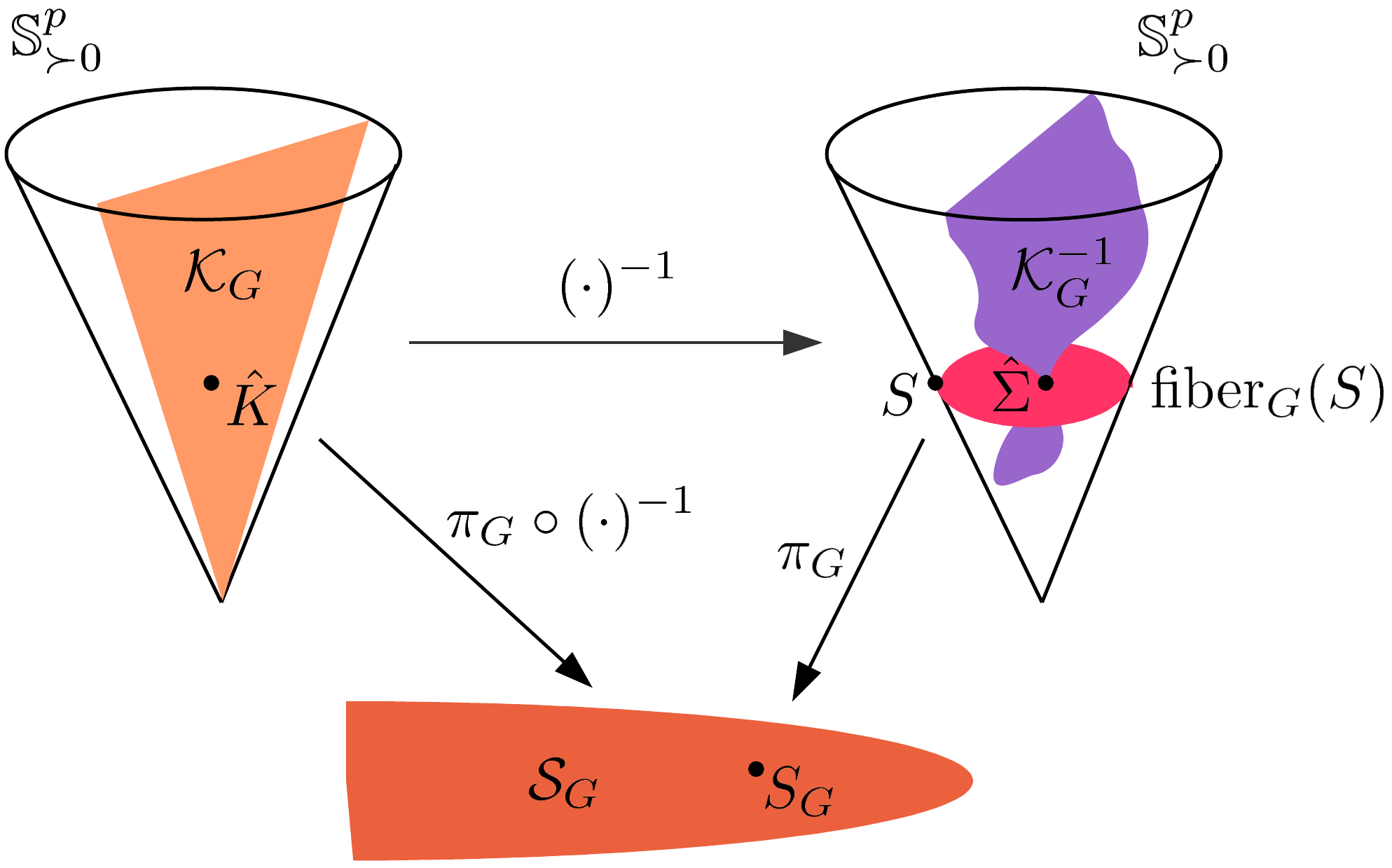}
\caption{\emph{Geometry of maximum likelihood estimation in Gaussian graphical models. The cone $\mathcal{K}_G$ consists of all concentration matrices in the model and $\mathcal{K}_G^{-1}$ is the corresponding set of covariance matrices. The cone of sufficient statistics $\mathcal{S}_G$ is defined as the projection of $\s^p_{\succeq 0}$ onto the (augmented) edge set $E^*$ of $G$. It is dual and homeomorphic to $\mathcal{K}_G$. Given a sample covariance matrix $S$, ${\rm fiber}_G(S)$ consists of all positive definite completions of the $G$-partial matrix $S_G$, and it intersects $\mathcal{K}_G^{-1}$ in at most one point, namely the MLE $\hat{\Sigma}$.}}
\label{fig_cones}
\end{figure}

\begin{theorem} \label{thm_fiber}
Consider a Gaussian model with linear constraints on the concentration matrix defined by $\mathcal{L}$ with $\mathcal{L}\cap\s^p_{\succ 0} \neq\emptyset$. Then the MLEs $\hat{\Sigma}$ and $\hat{K}$ exist for a given sample covariance matrix $S$ if and only if
$\,{\rm fiber}_\mathcal{L}(S)$ is non-empty, in which case $\,{\rm fiber}_\mathcal{L}(S)$ intersects $\mathcal{K}_\mathcal{L}^{-1}$ in exactly one point, namely the MLE $\,\hat{\Sigma}$. Equivalently, $\hat{\Sigma}$ is the unique maximizer of the determinant over the spectrahedron $\,{\rm fiber}_\mathcal{L}(S)$.
\end{theorem}

\begin{proof}
This proof is a simple exercise in convex optimization; see~\cite{boyd_Vandenberghe} for an introduction. The ML estimation problem for Gaussian models with linear constraints on the concentration matrix is given by 
\begin{equation*}
\label{opt_KL}
\begin{aligned}
& \underset{K}{\text{maximize}}
& &\log\det K - \tr(S K) \\
& \text{subject to}
& & \;K \in \mathcal{K}_{\mathcal{L}}.
\end{aligned}
\end{equation*}
Its dual is 
\begin{equation*}
\label{opt_dual_KL}
\begin{aligned}
& \underset{\Sigma}{\text{minimize}}
& &-\log\det \Sigma - p \\
& \text{subject to}
& &\Sigma\in {\rm fiber}_{\mathcal{L}}(S).
\end{aligned}
\end{equation*}
Since by assumption the primal problem is strictly feasible, strong duality holds by Slater's constraint qualification with the solutions satisfying $\hat{\Sigma} = \hat{K}^{-1}$. The MLE exists, i.e.~the global optimum of the two optimization problems is attained, if and only if the dual is feasible, i.e.~${\rm fiber}_{\mathcal{L}}(S)$ is non-empty. Let $\Sigma\in{\rm fiber}_\mathcal{L}(S)\cap\mathcal{K}_\mathcal{L}^{-1}$. Then $(\Sigma^{-1}, \Sigma)$ satisfies the KKT conditions, namely stationarity, primal and dual feasibility and complimentary slackness. Hence, this pair is primal and dual optimal. Thus, if $\,{\rm fiber}_\mathcal{L}(S)$ is non-empty, then $\,{\rm fiber}_\mathcal{L}(S)$ intersects $\mathcal{K}_\mathcal{L}^{-1}$ in exactly one point, namely the MLE $\,\hat{\Sigma}$, which is the dual optimal solution. This completes the proof.
\end{proof}

The geometry of ML estimation in Gaussian models with linear constraints on the concentration matrix is summarized in Figure~\ref{fig_cones} for the special case of Gaussian graphical models. The geometric picture for general linear concentration models is completely analogous. The convex geometry of Gaussian graphical models on 3 nodes is shown in Figure~\ref{fig_3d}. Since a general covariance matrix on 3 nodes lives in 6-dimensional space, we show the picture for correlation matrices instead, which live in 3-dimensional space.

\begin{figure}[t!]
\centering
\subfigure[$\mathcal{L}_G\cap\mathbb{S}^3_{\succeq 0}$]{\includegraphics[scale=0.29]{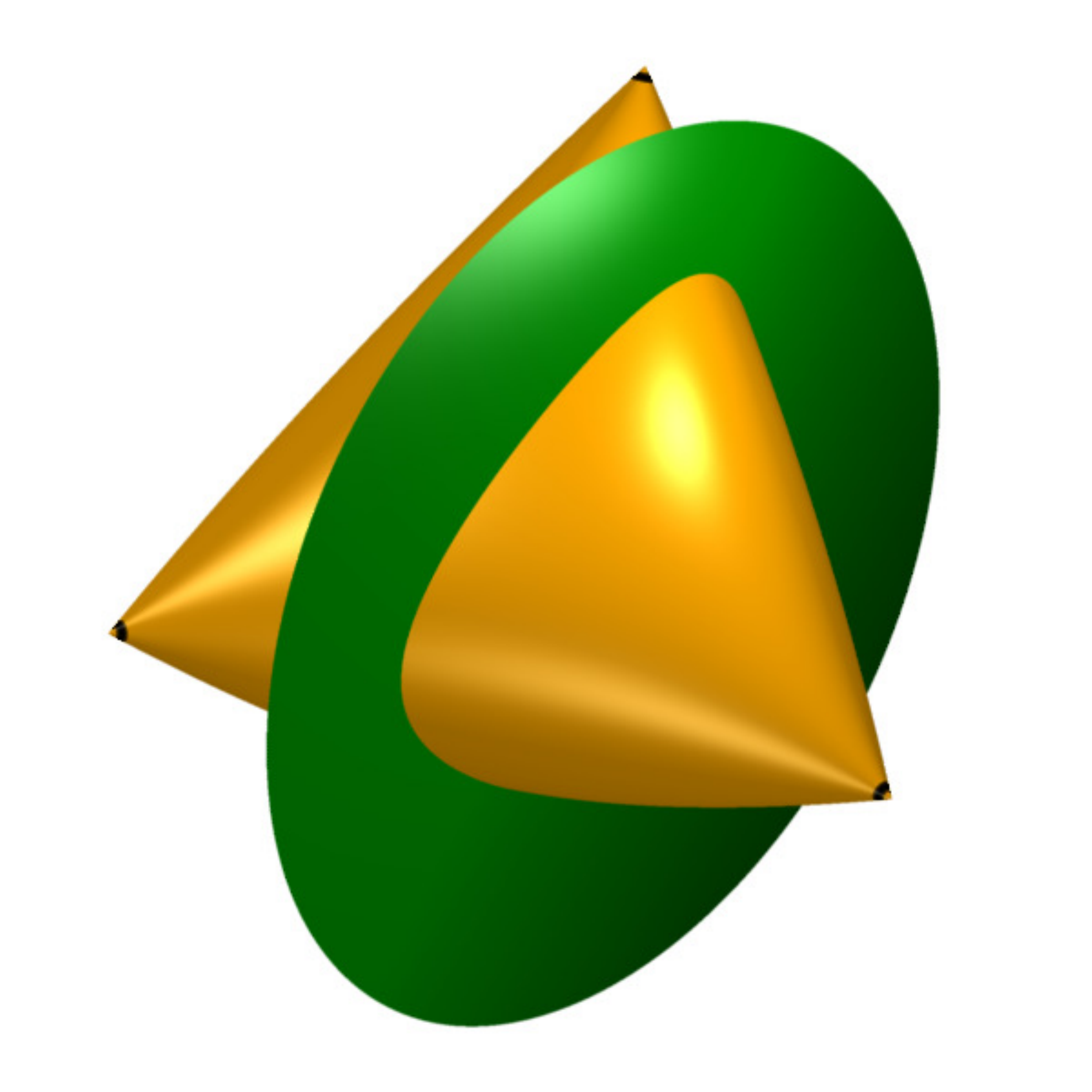} \label{fig_3d_a}}\qquad\qquad
\subfigure[$\mathcal{K}_G$]{\includegraphics[scale=0.29]{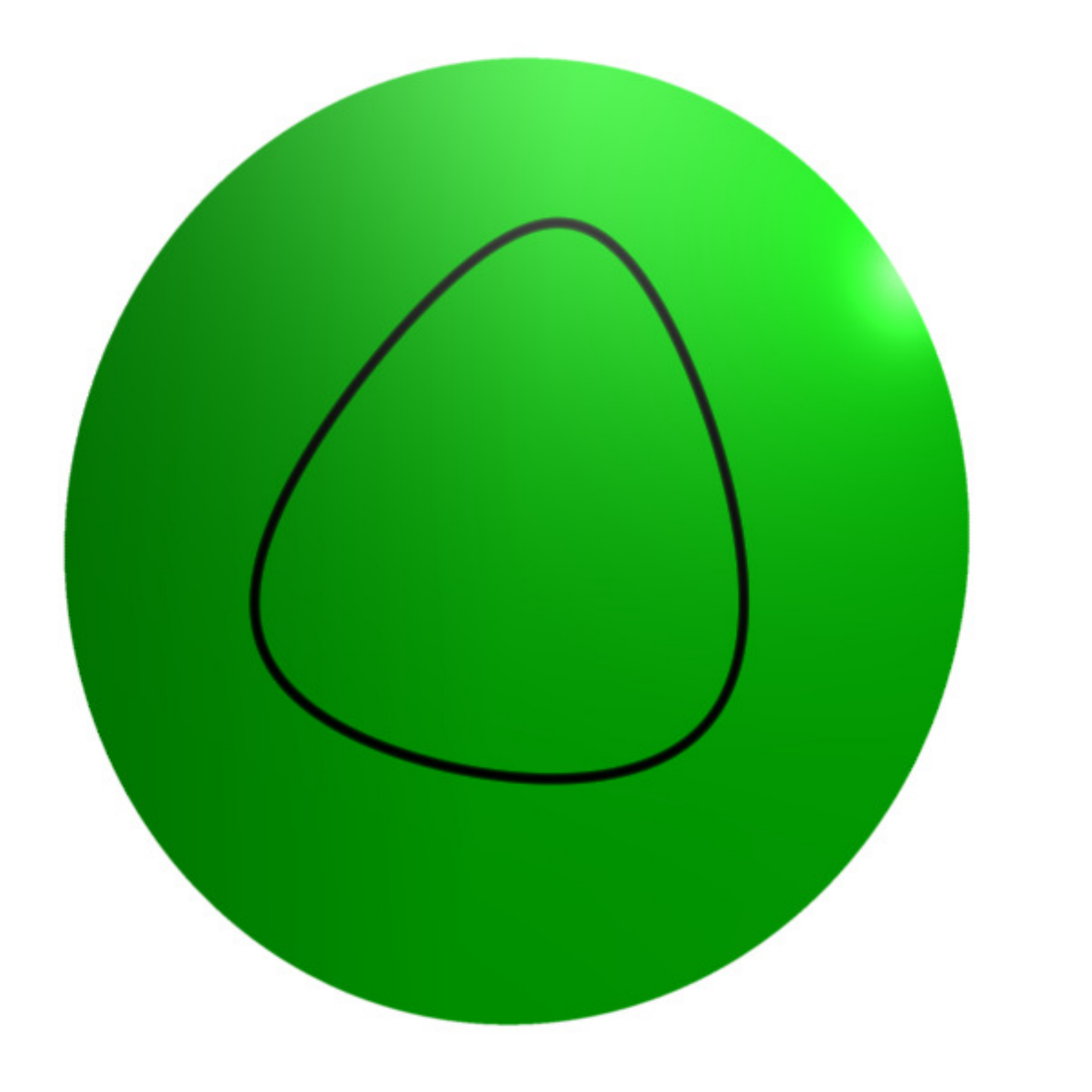} \label{fig_3d_b}}
\subfigure[$(\mathcal{K}_G)^{-1}$]{\includegraphics[scale=0.3]{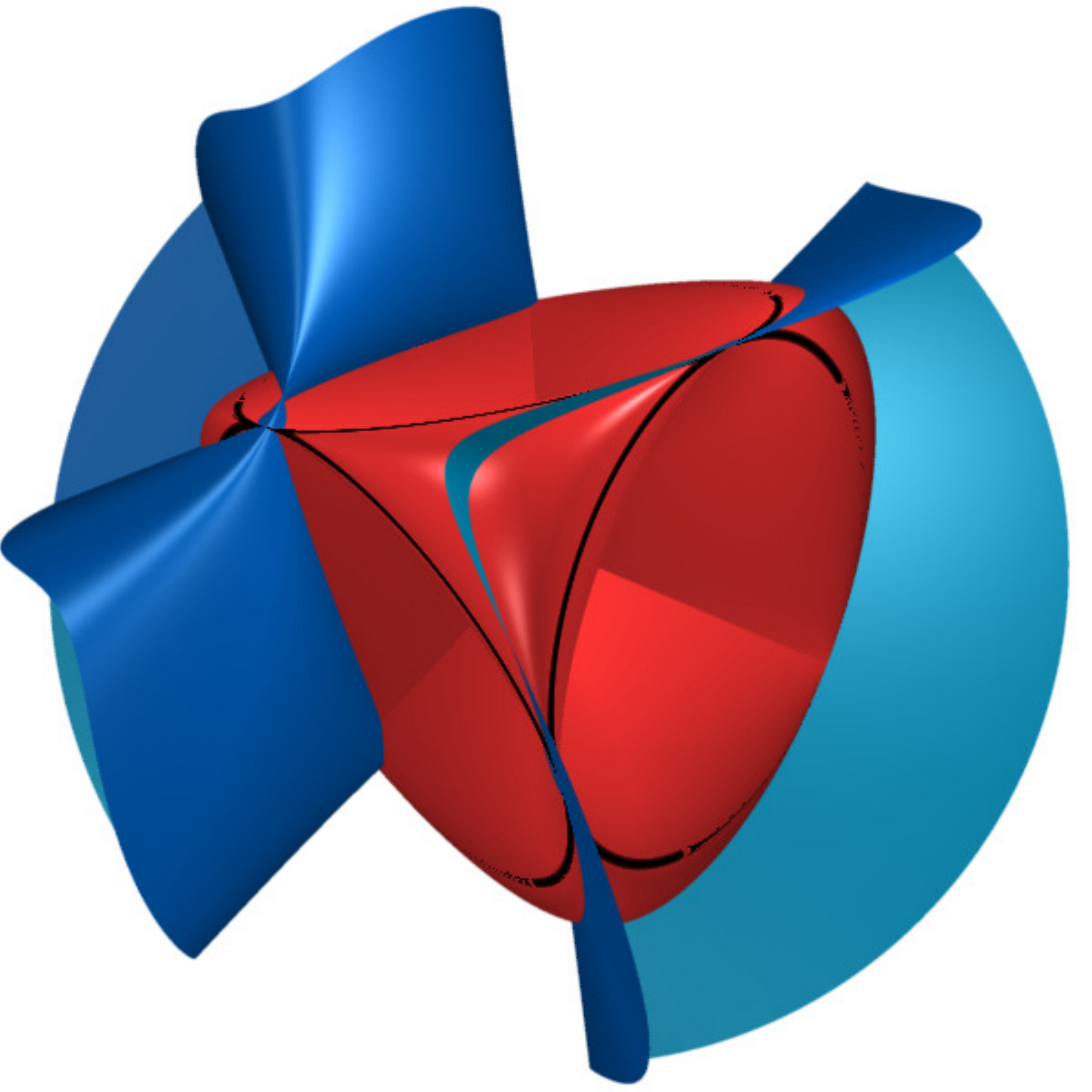} \label{fig_3d_c}}\qquad\qquad\quad
\subfigure[$\mathcal{S}_G$]{\includegraphics[scale=0.4]{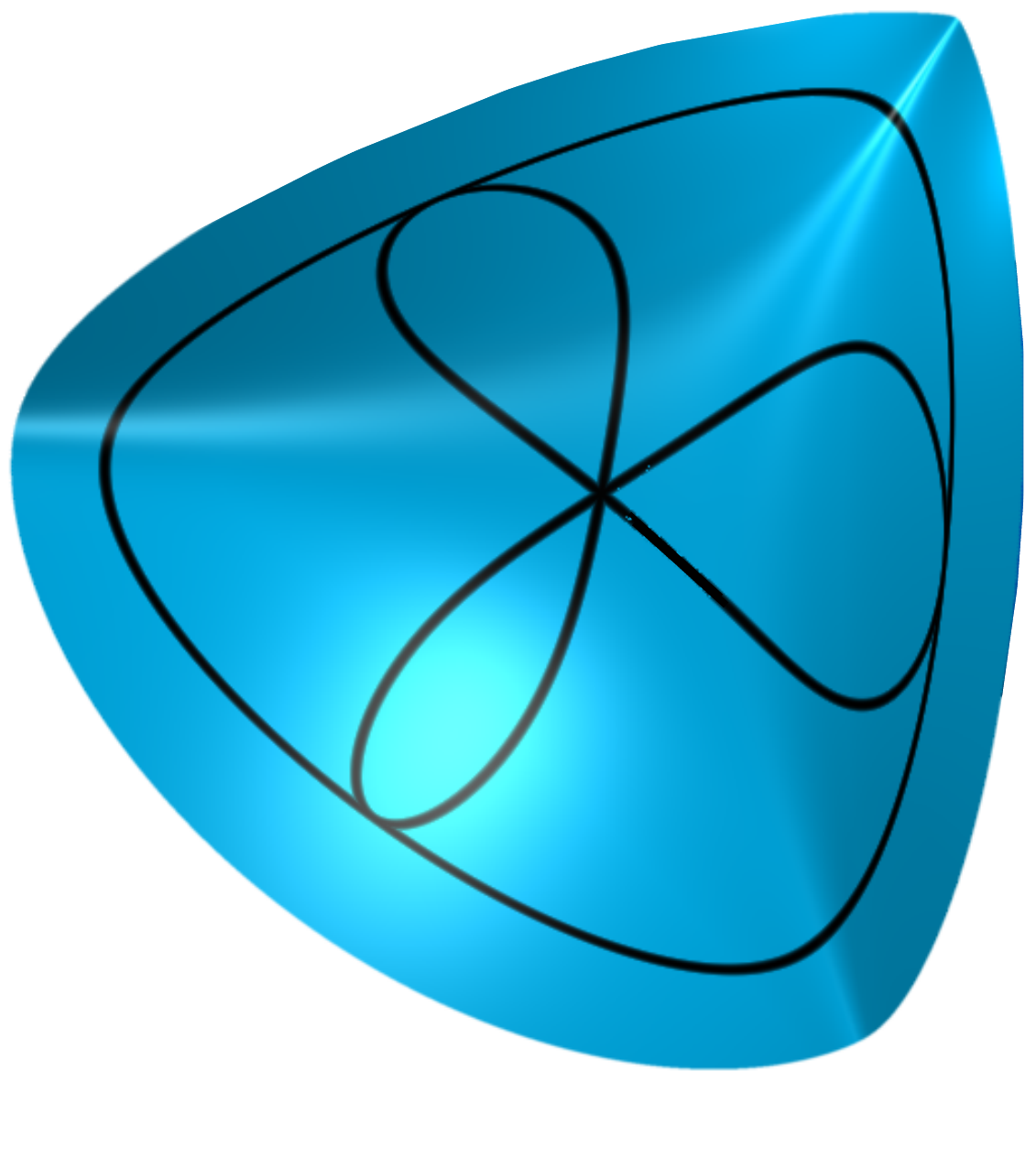}\qquad\quad \label{fig_3d_d}}
\caption{\emph{Geometry of Gaussian graphical models for $p=3$. The tetrahedral-shaped pillow in (a) corresponds to the set of all $3\times 3$ concentration matrices with ones on the diagonal. The linear subspace in (a) is defined by the missing edges in $G$. The resulting cone of concentration matrices is shown in (b). The corresponding set of covariance matrices is shown in (c), and the cone of sufficient statistics $\mathcal{S}_G$, dual to $\mathcal{K}_G$, is shown in (d).}}
\label{fig_3d}
\end{figure}

Theorem~\ref{thm_fiber} was first proven for Gaussian graphical models by Dempster~\cite{Dempster} and later more generally for regular exponential families in~\cite{Barndorff_1978, Brown_exponential}. One can show that the map 
$$\pi_G\circ (\cdot)^{-1}: \mathcal{K}_G\to\mathcal{S}_G$$ 
in Figure~\ref{fig_cones} corresponds to the gradient of the log-partition function. To embed this result into the theory of regular exponential families, we denote canonical parameters by $\theta$, minimal sufficient statistics by $t(X)$, and the log-partition function of a regular exponential family by $A(\theta)$. Then the theory of regular exponential families (see e.g.~\cite{Barndorff_1978, Brown_exponential}) implies that the gradient of the log-partition function $\nabla A(\cdot)$ defines a homeomorphism between the space of canonical parameters and the relative interior of the convex hull of sufficient statistics, and it is defined by $\nabla A(\theta) = \mathbb{E}_{\theta}(t(X))$. For Gaussian models we have $A(\theta) = \log\det(\theta)$; the algebraic structure in maximum likelihood estimation for Gaussian graphical models is a consequence of the fact that $\nabla A(\cdot)$ is a rational function. 

The geometric results and duality theory that hold for Gaussian graphical models can be extended to all regular exponential families~\cite{Barndorff_1978, Brown_exponential}. The algebraic picture can be extended to exponential families where $\nabla A(\cdot)$ is a rational function. This was shown in~\cite{exponential_varieties}, where it was proven that such exponential families are defined by \emph{hyperbolic polynomials}.


The problem of existence of the MLE can be studied at the level of sufficient statistics, i.e.~in the cone $\mathcal{S}_{G}$, or at the level of observations. As explained in Section~\ref{sec:pd_completion}, the MLE exists if and only if the sufficient statistics $S_G$ lie in the interior of the cone $\mathcal{S}_G$. Hence, analyzing existence of the MLE at the level of sufficient statistics requires analyzing the boundary of the cone $\mathcal{S}_G$. The boundary of $\mathcal{K}_G$ is defined by the hypersurface $\det(K)=0$ with $K_{i,j}=0$ for all $(i,j)\notin E^*$. It has been shown in~\cite{Sturmfels_Uhler} that the boundary of the cone $\mathcal{S}_G$ can be obtained by studying the dual of the variety defined by $\det(K)=0$. This algebraic analysis results in conditions that characterize existence of the MLE at the level of sufficient statistics. 

But perhaps more interesting from a statistical point of view, is a characterization of existence of the MLE at the level of observations. Note that if ${\rm rank}(S) < p$ then it can happen that ${\rm fiber}_{\mathcal{L}}(S)$ is empty,
in which case the MLE does not exist for $(\mathcal{L},S)$. In the next section, we discuss conditions on the number of observations $n$, or equivalently on the rank of $S$, that ensure existence of the MLE with probability 1 for particular classes of graphs.

%

\section{Existence of the MLE for various classes of graphs}
\label{sec:MLE_existence}


Since the Gaussian density is strictly positive, $\textrm{rank}(S)=\min(n,p)$ with probability 1. The \emph{maximum likelihood threshold} of a graph $G$, denoted $\textrm{mlt}(G)$, is defined as the minimum number of observations $n$ such that the MLE in the Gaussian graphical model with graph $G$ exists with probability 1. This is equivalent to the smallest integer $n$ such that for all generic positive semidefinite matrices $S$ of rank $n$ there exists a positive definite matrix $\Sigma$ with $S_G = \Sigma_G$. Although in this section we only consider Gaussian graphical models, note that this definition can easily be extended to general linear Gaussian concentration models.

The maximum likelihood threshold of a graph was introduced by Gross and Sullivant in~\cite{Gross_Sullivant}. Ben-David~\cite{Ben-David} introduced a related but different notion, the \emph{Gaussian rank} of a graph, namely the smallest $n$ such that the MLE exists for every positive semidefinite matrix $S$ of rank~$n$ for which every $n\times n$ principal submatrix is non-singular. Note that with probability 1 every $n\times n$ principal submatrix of a sample covariance matrix based on $n$ i.i.d.~samples from a Gaussian distribution is non-singular. Hence, the Gaussian rank of $G$ is an upper bound on $\textrm{mlt}(G)$. Since a sample covariance matrix of size $p\times p$ based on $n \leq p$ observations from a Gaussian population is of rank $n$ with probability 1, we here concentrate on the maximum likelihood threshold of a graph. 

A \emph{clique} in a graph $G$ is a completely connected subgraph of $G$. We denote by $q(G)$ the maximal clique-size of $G$. It is clear that the MLE cannot exist if $n<q(G)$, since otherwise the partial matrix $S_G$ would contain a completely specified submatrix that is not positive definite (the submatrix corresponding to the maximal clique). This results in a lower bound for the maximum likelihood threshold of a graph, namely 
$$\textrm{mlt}(G)\geq q(G).$$

For chordal graphs, Theorem~\ref{thm_chordal} shows that the MLE exists with probability 1 if and only if $n\geq q(G)$. Hence for chordal graphs it holds that $\textrm{mlt}(G)=q(G)$. However, this is not the case in general as shown by the following example.

\begin{example}
\label{ex_4cyle}
Let $G$ be the 4-cycle with edges $(1,2)$, $(2,3)$, $(3,4)$, and $(1,4)$. Then $q(G) =2$.  We define $X\in\mathbb{R}^{4\times 2}$ consisting of 2 samples in $\mathbb{R}^4$ and the corresponding sample covariance matrix $S=XX^T$ by
$$X=\begin{pmatrix} 1 & 0 \\ \frac{1}{\sqrt{2}} & \frac{1}{\sqrt{2}} \\ 0 & 1 \\ -\frac{1}{\sqrt{2}} & \frac{1}{\sqrt{2}}\end{pmatrix} \quad \textrm{and hence} \quad S=\begin{pmatrix} 1 & \frac{1}{\sqrt{2}} & 0 & -\frac{1}{\sqrt{2}} \\  \frac{1}{\sqrt{2}} & 1 & \frac{1}{\sqrt{2}} & 0 \\ 0 & \frac{1}{\sqrt{2}} & 1 & \frac{1}{\sqrt{2}} \\ -\frac{1}{\sqrt{2}} & 0 & \frac{1}{\sqrt{2}} & 1\end{pmatrix}.$$
One can check that $S_G$ cannot be completed to a positive definite matrix. In addition, there exists an open ball around $X$ for which the MLE does not exist. This shows that in general for non-chordal graphs $\textrm{mlt}(G)>q(G)$.
\end{example}

From Theorem~\ref{thm_chordal} we can determine an upper bound on $\textrm{mlt}(G)$ for general graphs. For a graph $G = (V, E)$ we denote by $G^+ = (V, E^+)$ a \emph{chordal cover} of $G$, i.e.~a chordal graph satisfying $E \subseteq E^+$. We denote the maximal clique size of $G^+$  by $q^+$. A \emph{minimal chordal cover}, denoted by $G^{\#} = (V, E^{\#})$, is a chordal cover of $G$, whose maximal clique size $q^{\#}$ achieves $q^{\#} = min(q^+)$ over all chordal covers of $G$. The quantity $q^{\#}(G)-1$ is also known as the \emph{treewidth} of $G$. It follows directly from Theorem~\ref{thm_chordal} that
$$\textrm{mlt}(G)\leq q^{\#}(G),$$
since if $S_{G^{\#}}$ can be completed to a positive definite matrix, so can $S_G$. 

If $G$ is a cycle, then $q(G)=2$ and $q^{\#}(G)=3$. Hence the MLE does not exist for $n=1$ and it exists with probability 1 for $n=3$. From Example~\ref{ex_4cyle} we can conclude that for cycles $\textrm{mlt}(G)=3$. Buhl~\cite{Buhl} shows that for $n=2$ the MLE exists with probability in $(0,1)$. More precisely, for $n=2$ we can view the two samples as vectors $x_1, \dots , x_p\in\mathbb{R}^2$. We denote by $\ell_1, \dots , \ell_p$ the lines defined by $x_1, \dots , x_p$. Then Buhl~\cite{Buhl} shows using an intricate trigonometric argument that the MLE for the $p$-cyle for $n=2$ exists if and only if the lines $\ell_1, \dots , \ell_p$ do not occur in one of the two sequences conforming with the ordering in the cycle $G$ as shown in Figure~\ref{cyle_conf}. In the following, we give an algebraic proof of this result by using an intriguing characterization of positive definiteness for $3\times 3$ symmetric matrices given in~\cite{Barrett1}.

\begin{figure}[t!]
\centering
\subfigure[$p$-cycle]{\includegraphics[scale=0.31]{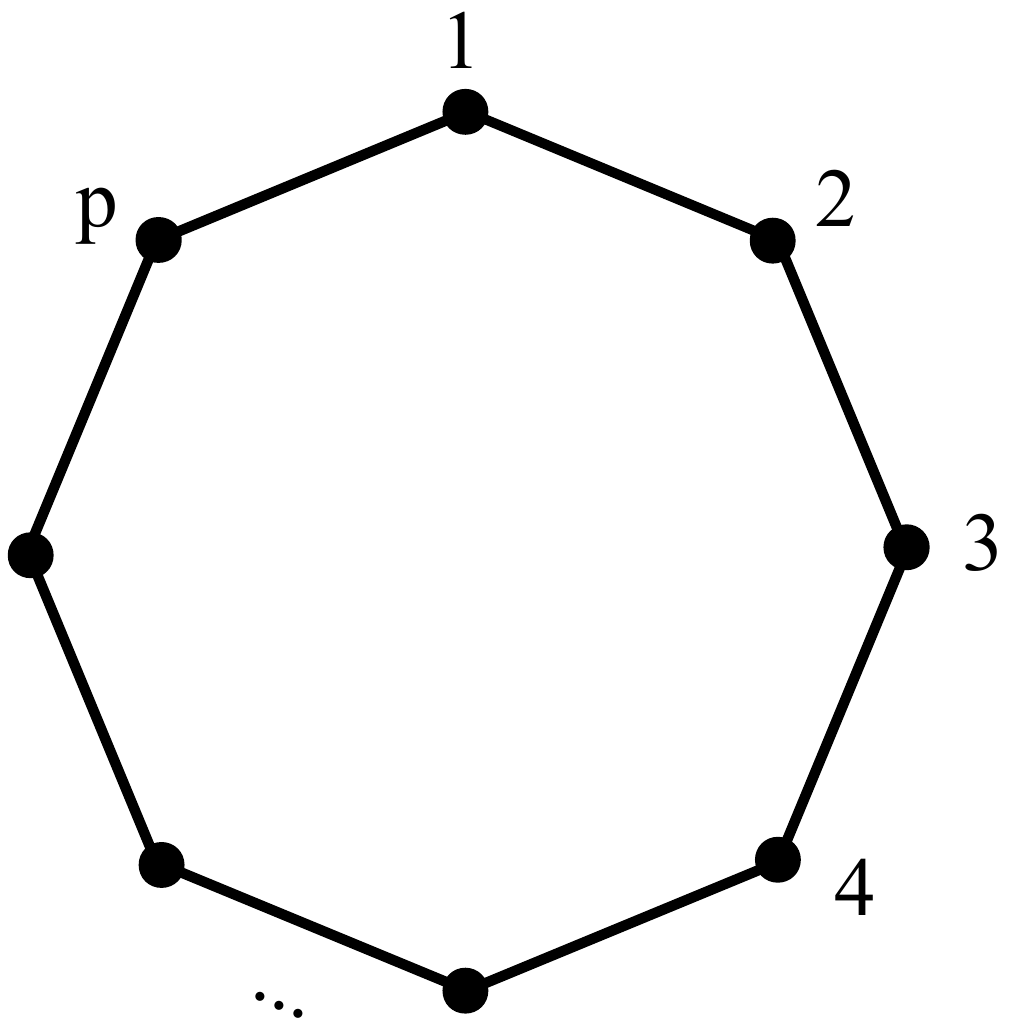}}\qquad\qquad
\subfigure[Line configurations for which the MLE does not exist.]{\includegraphics[scale=0.39]{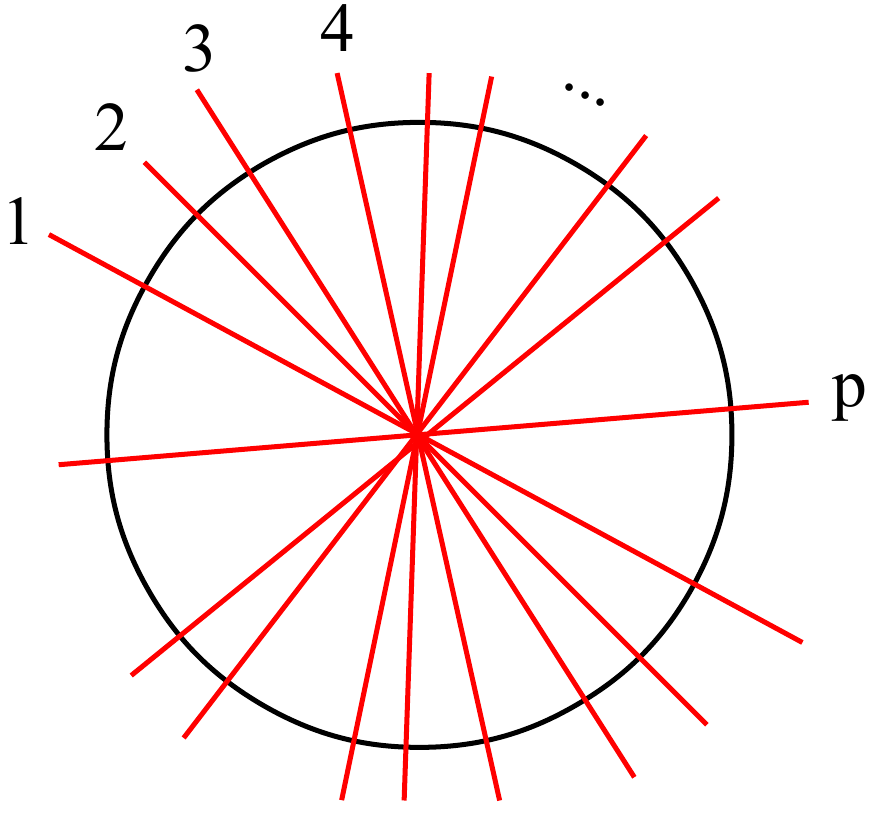}\qquad\qquad\quad \includegraphics[scale=0.39]{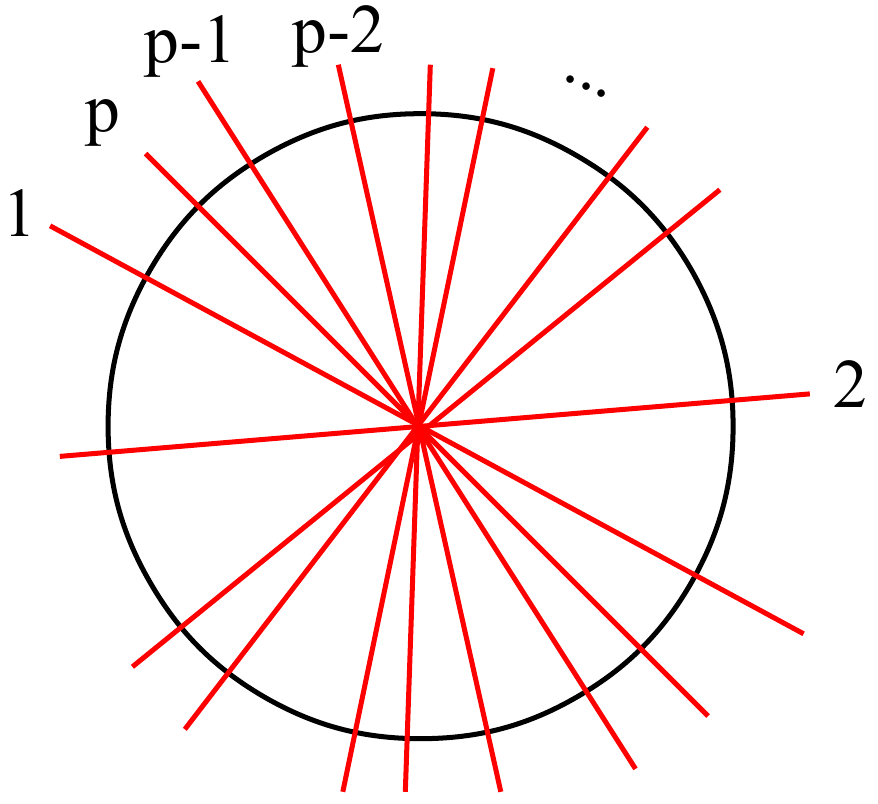}}
\caption{\emph{Buhl's geometric criterion~\cite{Buhl}  for existence of the MLE for $n=2$ in a Gaussian graphical model on the $p$-cycle.}}
\label{cyle_conf}
\end{figure}

\begin{proposition}[Barrett et al.~\cite{Barrett1}]
The matrix
$$\begin{pmatrix} 1 & \cos(\alpha) & \cos(\beta) \\ \cos(\alpha) & 1 & \cos(\gamma) \\ 
\cos(\beta) & \cos(\gamma) & 1\end{pmatrix}$$ with $0 < \alpha, \beta, \gamma < \pi$ is positive definite if and only if
$$\alpha< \beta + \gamma, \quad \beta< \alpha + \gamma, \quad \gamma< \alpha + \beta, \quad \alpha+ \beta + \gamma < 2\pi.$$
\end{proposition}

Let $G$ denote the $p$-cycle. Then, as shown in~\cite{Barrett1}, this result can be used to give a characterization for completability of a $G$-partial matrix to a positive definite matrix through induction on the cycle length $p$.

\begin{corollary}[Barrett et al.~\cite{Barrett1}]
\label{cor_cycle}
Let $G$ be the $p$-cycle. Then the $G$-partial matrix
$$\begin{pmatrix} 1 & \cos(\theta_{1}) & & & \cos(\theta_{p})  \\  \cos(\theta_{1}) & 1 & \cos(\theta_{2}) & ? & \\ & \cos(\theta_{2}) & 1 & & & \\  & ? & & \ddots & \cos(\theta_{p-1}) \\  \cos(\theta_{p}) & & & \cos(\theta_{p-1})& 1 \end{pmatrix}$$
with $0 < \theta_{1}, \theta_{2}, \dots \theta_{p} < \pi$ has a positive definite completion if and only if for each $S\subseteq [p]$ with $|S|$ odd,
$$\sum_{i\in S} \theta_i < (|S|-1)\pi + \sum_{j\notin S} \theta_j.$$
\end{corollary}

Buhl's result~\cite{Buhl} can easily be deduced from this algebraic result about the existence of positive definite completions: For $n=2$ we view the observations as vectors $x_1, \dots , x_p\in\mathbb{R}^2$. Note that we can rescale and rotate the data vectors $x_1,\dots , x_{p}$ (i.e.~perform an orthogonal transformation) without changing the problem of existence of the MLE. So without loss of generality we can assume that the vectors $x_1,\dots , x_{p}\in\R^2$ have length one, lie in the upper unit half circle, and $x_1=(1,0)$. Now we denote by $\theta_i$ the angle between $x_i$ and $x_{i+1}$, where $x_{p+1}:=x_1$. One can show that the angular conditions in Corollary~\ref{cor_cycle} are equivalent to requiring that the vectors $x_1,\dots , x_{p}\in\R^2$ do not occur in one of the two sequences conforming with the ordering in the cycle $G$ as shown in Figure~\ref{cyle_conf}.

Hence, for a $G$-partial matrix to be completable to a positive definite matrix, it is necessary that every submatrix corresponding to a clique in the graph is positive definite and every partial submatrix corresponding to a cycle in $G$ satisfies the conditions in Corollary~\ref{cor_cycle}. Barrett et al.~\cite{Barrett2} characterized the graphs for which these conditions are sufficient for existence of a positive definite completion. They showed that this is the case for graphs that have a chordal cover with no new 4-cliques. Such graphs can be obtained as a \emph{clique sum} of chordal graphs and series-parallel graphs (i.e.~graphs $G$ with $q^{\#}(G)\leq 3$)~\cite{Johnson_McKee}. To be more precise, for such graphs $G=(V,E)$ the vertex set can be decomposed into three disjoint subsets $V=V_1\cup V_2\cup V_3$ such that there are no edges between $V_1$ and $V_3$, the subgraph induced by $V_2$ is a clique, and the subgraphs induced by $V_1\cup V_2$ and $V_2\cup V_3$ are either chordal or series-parallel graphs or can themselves be decomposed as a clique sum of chordal or series-parallel graphs. For such graphs it follows that 
$$\textrm{mlt}(G)=\max(3, q(G))=q^{\#}(G).$$ 
This raises the question whether there exist graphs for which $\textrm{mlt}(G)<q^{\#}(G)$, i.e., graphs for which the MLE exists with probability 1 even if the number of observations is strictly smaller than the maximal clique size in a minimal chordal cover of $G$. This question has been answered to the positive for $3\times 3$ grids using an algebraic argument in~\cite{Uhler_2012} and more generally for grids of size $m\times m$ using a combinatorial argument in~\cite{Gross_Sullivant}. In particular, let $G$ be a grid of size $m\times m$. Then $q^{\#}(G)=m+1$, but it was shown in~\cite{Gross_Sullivant} that the MLE exists with probability 1 for $n=3$, independent of the grid size $m$. Grids are a special class of planar graphs. Gross and Sullivant~\cite{Gross_Sullivant} more generally proved that for any planar graph it holds that $\textrm{mlt}(G)\leq 4$.

%
%
%

\section{Algorithms for computing the MLE}
\label{sec:MLE_alg}

After having discussed when the MLE exists, we now turn to the question of how to compute the MLE for Gaussian graphical models. As described in Section~\ref{sec:likelihood}, determining the MLE in a Gaussian model with linear constraints on the inverse covariance matrix is a convex optimization problem. Hence, it can be solved in polynomial time for instance using interior point methods~\cite{boyd_Vandenberghe}. These are implemented for example in~\verb+cvx+, a user-friendly \verb+matlab+ software for disciplined convex programming~\cite{cvx}. 

\begin{algorithm}[!b]
\caption{Coordinate descent on $\Sigma$}
\label{algthm:IPS_opt_Z}
\begin{algorithmic}
\begin{STATE}
\vspace{0.2cm}
{\bf Input:\;\;\;}
Graph $G=(V,E)$, sample covariance matrix $S$, and precision $\epsilon$.

{\bf Output:} 
MLE $\hat{\Sigma}$.

\vspace{0.2cm}

\begin{enumerate}

\item[(1)] Let $\Sigma^0 = S$
\item[(2)] Cycle through $(u,v)\notin E^*$ and solve the following optimization problem:
\begin{equation*}
\begin{aligned}
& \underset{\Sigma\succeq 0}{\text{maximize}}
& & \log\det(\Sigma)   \\
& \text{subject to}
&& \Sigma_{i,j} = \Sigma^0_{i,j} \;\textrm{ for all } (i,j)\neq (u,v).
\end{aligned}
\end{equation*}
and update $\Sigma^1 := \Sigma$.
\item[(3)] If $|\!|\Sigma^0-\Sigma^1|\!|_1 <\epsilon$, let $\hat{\Sigma}:=\Sigma^1$. Otherwise, let $\Sigma^0:=\Sigma^1$ and return to (2).
\end{enumerate}
\end{STATE}
\end{algorithmic}
\end{algorithm}

Although interior point methods run in polynomial time, for very large Gaussian graphical models it is usually more practical to apply coordinate descent algorithms. The idea of using coordinate descent algorithms for computing the MLE in Gaussian graphical models was already present in the original paper by Dempster~\cite{Dempster}. Coordinate descent on the entries of $\Sigma$ was first implemented by Wermuth and Scheidt~\cite{Wermuth_Scheidt} and is shown in Algorithm~\ref{algthm:IPS_opt_Z}. In this algorithm, we start with $\Sigma^0 = S$ and iteratively update the entries $(i,j)\notin E^*$ by maximizing the log-likelihood in direction $\Sigma_{i,j}$ and keeping all other entries fixed. 

Note that step (2) in Algorithm~\ref{algthm:IPS_opt_Z} can be given in closed-form: Let $A=\{u,v\}$ and $B=V\setminus A$. We now show that the objective function in step (2) of Algorithm~\ref{algthm:IPS_opt_Z} can be written in terms of the $2\times 2$ Schur complement $\Sigma' = \Sigma_{A,A}-\Sigma_{A,B}\Sigma_{B,B}^{-1}\Sigma_{B,A}$. To do this, note that $\det (\Sigma) = \det(\Sigma')\det(\Sigma_{B,B})$. Since $\Sigma_{B,B}$ is held constant in the optimization problem, then up to an additive constant it holds that
$$\log\det (\Sigma)  = \log\det (\Sigma').$$
Thus, the optimization problem in step (2) of Algorithm~\ref{algthm:IPS_opt_Z} is equivalent to
\begin{equation*}
\begin{aligned}
& \underset{\Sigma'\succeq 0}{\text{maximize}}
& & \log\det(\Sigma')\\
& \text{subject to}
&& \Sigma'_{i,i}  = \Sigma^0_{i,i}-\Sigma^0_{i,B}(\Sigma^0_{B,B})^{-1}\Sigma^0_{B,i}, \;\; i\in A,\\
\end{aligned}
\end{equation*}
and the global maximum is attained by $\Sigma'_{u,v} = 0$. Hence, the solution to the univariate optimization problem in step (2) of Algorithm~\ref{algthm:IPS_opt_Z} is
$$\Sigma_{u,v} = \Sigma_{u,B}\Sigma_{B,B}^{-1}\Sigma_{B,v},$$
forcing the corresponding entry of $\Sigma^{-1}$ to be equal to zero. 

Dual to this algorithm, one can define an equivalent algorithm that cycles through entries of the concentration matrix corresponding to $(i,j)\in E$, starting in the identity matrix. This procedure is shown in Algorithm~\ref{algthm:IPS_opt_K}. Similarly as for Algorithm~\ref{algthm:IPS_opt_Z}, the solution to the optimization problem in step (2) can be given in closed-form. Defining as before, $A=\{u,v\}$ and $B=V\setminus A$, then analogously as in the derivation above, one can show that the solution to the optimization problem in step (2) of Algorithm~\ref{algthm:IPS_opt_K} is
$$K_{A,A} = (S_{A,A})^{-1} + K_{A,B}K_{B,B}^{-1}K_{B,A},$$
forcing $\Sigma_{A,A}$ to be equal to $S_{A,A}$. This algorithm, which tries to match the sufficient statistics, is analogous to \emph{iterative proportional scaling} for computing the MLE in contingency tables~\cite{Haberman}. Convergence proofs for both algorithms were given by Speed and Kiiveri~\cite{Speed_Kiiveri}.

\begin{algorithm}[!t]
\caption{Coordinate descent on $K$}
\label{algthm:IPS_opt_K}
\begin{algorithmic}
\begin{STATE}
\vspace{0.2cm}
{\bf Input:\;\;\;}
Graph $G=(V,E)$, sample covariance matrix $S$, and precision $\epsilon$.

{\bf Output:} 
MLE $\hat{K}$.

\vspace{0.2cm}

\begin{enumerate}

\item[(1)] Let $K^0 = \textrm{Id}$.
\item[(2)] Cycle through $(u,v)\in E$ and solve the following optimization problem:
\begin{equation*}
\begin{aligned}
& \underset{K \succeq 0}{\text{maximize}}
& & \log\det(K) - \textrm{trace}(KS)  \nonumber\\
& \text{subject to}
&& K_{i,j} = K^0_{i,j} \;\textrm{ for all } (i,j)\in (V\times V)\setminus\{(u,u),(v,v),(u,v)\}.
\end{aligned}
\end{equation*}
and update $K^1 := K$.
\item[(3)] If $|\!|K^0-K^1|\!|_1 <\epsilon$, let $\hat{K}:=K^1$. Otherwise, let $K^0:=K^1$ and return to (2).
\end{enumerate}
\end{STATE}
\end{algorithmic}
\end{algorithm}


In general the MLE must be computed iteratively. However, in some cases estimation can be made in closed form. A trivial case when the MLE of a Gaussian graphical model can be given explicitly is for complete graphs: In this case, assuming that the MLE exists, i.e.~$S$ is non-singular, then $\hat{K} = S^{-1}$. In~\cite[Section 5.3.2]{lauritzen1996}, Lauritzen showed that also for chordal graphs the MLE has a closed-form solution. This result is based on the fact that any chordal graph $G=(V,E)$ is a clique sum of cliques, i.e., the vertex set can be decomposed into three disjoint subsets $V=A\cup B\cup C$ such that there are no edges between $A$ and $C$, the subgraph induced by $B$ is a clique, and the subgraphs induced by $A\cup B$ and $B\cup C$ are either cliques or can themselves be decomposed as a clique sum of cliques. In such a decomposition, $B$ is known as a \emph{separator}. In~\cite[Proposition 5.9]{lauritzen1996}, Lauritzen shows that, assuming existence of the MLE, then the MLE for a chordal Gaussian graphical model is given by
\begin{equation}
\label{eq_MLE_chordal}
\hat{K} = \sum_{C\in\mathcal{C}} \left[(S_{C,C})^{-1}\right]^{\textrm{fill}} - \sum_{B\in\mathcal{B}} \left[(S_{B,B})^{-1}\right]^{\textrm{fill}},
\end{equation}
where $\mathcal{C}$ denotes the maximal cliques in $G$, $\mathcal{B}$ denotes the separators in the clique decomposition of $G$ (with multiplicity, i.e., a clique could appear more than once), and $[A_{HH}]^{\textrm{fill}}$ denotes a $p\times p$ matrix, where the submatrix corresponding to $H\subset V$ is given by $A$ and all the other entries are filled with zeros. 

To gain more insight into the formula~(\ref{eq_MLE_chordal}), consider the simple case where the subgraphs corresponding to $A\cup B$ and $B\cup C$ are cliques. Then (\ref{eq_MLE_chordal}) says that the MLE is given by
\begin{equation}
\label{eq_chordal_small}
\hat{K} = \left[S_{1}^{-1}\right]^{\textrm{fill}} +  \left[S_{2}^{-1}\right]^{\textrm{fill}} - \left[S_{B}^{-1}\right]^{\textrm{fill}},
\end{equation}
where we simplified notation by setting $S_1 = S_{AB,AB}$, $S_2=S_{BC,BC}$, and $S_B=S_{B,B}$, also to clarify that we first take the submatrix and then invert it.
To prove~(\ref{eq_chordal_small}), it suffices to show that $\left(\hat{K}^{-1}\right)_G = S_G$, since $\hat{K}_{i,j}=0$ for all $(i,j)\notin E^*$. We first expand $\hat{K}$ and then use Schur complements to compute its inverse:
\begin{equation}
\label{eq_cliques}
\hat{K} = \begin{pmatrix}
\left(S_{1}^{-1}\right)_{A,A} & \left(S_{1}^{-1}\right)_{A,B} & 0 \\  \left(S_{1}^{-1}\right)_{B,A} & \left(S_{1}^{-1}\right)_{B,B} +\left(S_{2}^{-1}\right)_{B,B} - S_{B}^{-1} & \left(S_{2}^{-1}\right)_{B,C} \\ 0 & \left(S_{2}^{-1}\right)_{C,B} & \left(S_{2}^{-1}\right)_{C,C}
\end{pmatrix}.
\end{equation}
Denoting $\hat{K}^{-1}$ by $\hat{\Sigma}$ and using Schur complements, we obtain
$$\hat{\Sigma}_{AB,AB} = 
\begin{pmatrix}
\left(S_{1}^{-1}\right)_{A,A} & \left(S_{1}^{-1}\right)_{A,B} \\  \left(S_{1}^{-1}\right)_{B,A} & \left(S_{1}^{-1}\right)_{B,B} +\left(S_{2}^{-1}\right)_{B,B} - S_{B}^{-1} - \left(S_{2}^{-1}\right)_{B,C} \left((S_{2}^{-1})_{C,C}\right)^{-1} \left(S_{2}^{-1}\right)_{C,B}.
\end{pmatrix}^{-1}$$
Note that by using Schur complements once again,
$$\left(S_{2}^{-1}\right)_{B,B} - \left(S_{2}^{-1}\right)_{B,C} \left((S_{2}^{-1})_{C,C}\right)^{-1} \left(S_{2}^{-1}\right)_{C,B} = S_{B}^{-1},$$
and hence $\hat{\Sigma}_{AB,AB} = S_1$. Analogously, it follows that $\hat{\Sigma}_{BC,BC} = S_2$, implying that $\hat{\Sigma}_G = S_G$. The more general formula for the MLE of chordal Gaussian graphical models in (\ref{eq_MLE_chordal}) is obtained by induction and repeated use of (\ref{eq_cliques}).

A stronger property than existence of a closed-form solution for the MLE is to ask which Gaussian graphical models have rational formulas for the MLE in terms of the entries of the sample covariance matrix. An important observation is that the number of critical points to the likelihood equations is constant for generic data, i.e., it is constant with probability~1 (it can be smaller on a measure zero subspace). The number of solutions to the likelihood equations for generic data, or equivalently, the maximum number of solutions to the likelihood equations, is called the \emph{maximum likelihood degree} (\emph{ML degree}). Hence, a model has a rational formula for the MLE if and only if it has ML degree 1. It was shown in~\cite{Sturmfels_Uhler} that the ML degree of a Gaussian graphical model is 1 if and only if the underlying graph is chordal. The ML degree of the 4-cycle can easily be computed and is known to be~5; see~\cite[Example~2.1.13]{Oberwolfach} for some code on how to do the computation using the open-source computer algebra system \verb+Singular+~\cite{Singular}. It is conjectured in~\cite[Section 7.4]{Oberwolfach} that the ML degree of the cycle grows exponentially in the cycle length, namely as $(p-3)2^{p-2}+1$, where $p\geq 3$ is the cycle length.

Since the likelihood function is strictly concave for Gaussian graphical models, this implies that even when the ML degree is larger than 1, there is still a unique local maximum of the likelihood function. As a consequence, while there are multiple complex solutions to the ML equations for non-chordal graphs, there is always a unique solution that is real and results in a positive definite matrix.

\section{Learning the underlying graph}
\label{sec:graph_learning}

Until now we have assumed that the underlying graph is given to us. In this section, we present methods for learning the underlying graph. We here only provide a short overview of some of the most prominent methods for model selection in Gaussian graphical models; for more details and for practical examples, see~\cite{graphical_models_R}.

A popular method for performing model selection is to take a stepwise approach. We start in the empty graph (or in the complete graph) and run a forward search (or a backward search). We cycle through the possible edges and add an edge (or remove an edge) if it decreases some criterion. Alternatively, one can also search for the edge which minimizes some criterion and add (or remove) this edge, but this is considerably slower. Two popular objective functions are the \emph{Akaike information criterion} (AIC) and the \emph{Bayesian information criterion} (BIC)~\cite{AIC, BIC}. These criteria are based on penalizing the likelihood according to the model complexity, i.e.
\begin{equation}
\label{eq_penalty}
-2\ell +\lambda |E|,
\end{equation}
where $\ell$ is the log-likelihood function, $\lambda$ is a parameter that penalizes model complexity, and $|E|$ denotes the number of edges, or equivalently, the number of parameters in the model. The AIC is defined by choosing $\lambda=2$, whereas the BIC is defined by setting $\lambda=\log(n)$ in~(\ref{eq_penalty}). 

Alternatively, one can also use significance tests for testing whether a particular partial correlation is zero and removing the corresponding  edge accordingly. A hypothesis test for zero partial correlation can be built based on Fisher's z-transform~\cite{Fisher15}: For testing whether $K_{i,j}=0$, let $A=\{i,j\}$ and $B=V\setminus A$. In Proposition~\ref{prop_Gaussian} we saw that $K_{A,A}^{-1} = \Sigma_{A\mid B}$. Hence testing whether $K_{i,j}=0$ is equivalent to testing whether the correlation $\rho_{i,j\mid B}$ is zero. The sample estimate of $\rho_{i,j\mid B}$ is given by
$$\hat{\rho}_{i,j\mid B} = S_{i,j} - S_{i,B} S_{B,B}^{-1} S_{B,j}.$$ 
Fisher's z-transform is defined by 
$$\hat{z}_{i,j \mid B} = \frac{1}{2} \log \biggr (  \frac{1+\hat{\rho}_{i,j \mid B}}{1-\hat{\rho}_{i,j \mid B}}\biggr).$$ 
Fisher~\cite{Fisher15} showed that using the test statistic $T_n = \sqrt{n - p+2 - 3 |\hat{z}_{i,j\mid B}|}$ with a rejection region $R_n  = (-\Phi^{-1}(1 -\alpha/2), \Phi^{-1}(1 -\alpha/2))$, where $\Phi$ denotes the cumulative distribution function of $\mathcal{N}(0,1)$, leads to a test of size $\alpha$.

A problem with stepwise selection strategies is that they are impractical for large problems or only a small part of the relevant search space can be covered during the search. A simple alternative, but a seemingly naive method for model selection in Gaussian graphical models, is to set a specific threshold for the partial correlations and remove all edges corresponding to the partial correlations that are less than the given threshold. This often works well, but a disadvantage is that the resulting estimate of the inverse covariance matrix might not be positive definite. 

An alternative is to use the \emph{glasso} algorithm~\cite{Friedman_2008}. It is based on maximizing the $\ell_1$-penalized log-likelihood function, i.e.
$$\ell_{\textrm{pen}}(K) = \log \det (K) - \textrm{tr}(KS) -\lambda |K|_1,$$
where $\lambda$ is a non-negative parameter that penalizes model complexity and $|K|_1$ is the sum of the absolute values of the off-diagonal elements of the concentration matrix. The use of $|K|_1$ is a convex proxy for the number of non-zero elements of $K$ and allows efficient optimization of the penalized log-likelihood function by convex programming methods such as interior point algorithms or coordinate descent approaches similar to the ones discussed in Section~\ref{sec:MLE_alg}; see e.g.~\cite{Mazumder_Hastie}. A big advantage of using  $\ell_1$-penalized maximum likelihood estimation for model selection in Gaussian graphical models is that it can also be applied in the high-dimensional setting and comes with structural recovery guarantees~\cite{Ravikumar_2011}. Various alternative methods for learning high-dimensional Gaussian graphical models have been proposed that have similar guarantees, including node-wise regression with the lasso~\cite{Meinshausen_Buehlmann}, a constrained $\ell_1$-minimization approach for inverse matrix estimation (CLIME)~\cite{Cai_CLIME}, and a testing approach with false discovery rate control~\cite{Liu_2013}. 

\section{Other Gaussian models with linear constraints}
\label{sec:linear_Gaussian}

Gaussian graphical models are Gaussian models with particular equality constraints on the concentration matrix, namely where some of the entries are set to zero. We end by giving an overview on other Gaussian models with linear constraints.

Gaussian graphical models can be generalized by introducing a vertex and edge coloring: Let $G=(V, E)$ be an undirected graph, where the vertices are colored with $s$ different colors and the edges with $t$ different colors. This leads to a partition of the vertex and edge set into color classes, namely,
$$V =V_1 \cup V_2 \cup V_s, \;\;s\leq p, \quad \textrm{and} \quad E = E_1 \cup E_2 \cup \cdots \cup E_t, \;\;t \leq |E|.$$
An \emph{RCON} model on $G$ is a Gaussian graphical model on $G$ with some additional equality constraints, namely that $K_{i,i}=K_{j,j}$ if $i$ and $j$ are in the same vertex color class and $K_{i,j} = K_{u,v}$ if $(i,j)$ and $(u,v)$ are in the same edge color class. Hence a Gaussian graphical model on a graph $G$ is an RCON model on $G$, where each vertex and edge has a separate color.  

Determining the MLE for RCON models leads to a convex optimization problem and the corresponding dual optimization problem can be readily computed:
 \begin{equation*}
\label{opt_dual_RCON}
\begin{aligned}
& \underset{\Sigma\succeq 0}{\text{minimize}}
& &\;\,-\log\det \Sigma - p &\\
& \text{subject to}
& & \;\,\,\sum_{\alpha\in V_i} \Sigma_{\alpha,\alpha} = \sum_{\alpha\in V_i} S_{\alpha,\alpha}, & \textrm{for all } 1\leq i\leq s,\\  
& && \sum_{(\alpha,\beta)\in E_j} \Sigma_{\alpha,\beta} = \sum_{(\alpha, \beta)\in E_j} S_{\alpha,\beta}, & \textrm{for all } 1\leq j \leq t.  
\end{aligned}
\end{equation*}
This shows that the constraints for existence of the MLE in an RCON model on a graph~$G$ are relaxed as compared to a Gaussian graphical model on $G$; namely, in an RCON model the constraints are only on the sum of the entries in a color class, whereas in a Gaussian graphical model the constraints are on each entry.

RCON models were introduced by H{\o}jsgaard and Lauritzen in~\cite{RCON}. These models are useful for applications, where symmetries in the underlying model can be assumed. Adding symmetries reduces the number of parameters and in some cases also the number of observations needed for existence of the MLE. For example, defining $G$ to be the 4-cycle and having only one vertex color class and one edge color class (i.e., we color each vertex in the same color and each edge in the same color), then one can show that the MLE already exists for 1 observation with probability 1. This is in contrast to the result that $\textrm{mlt}(G)=3$ for cycles as shown in Section~\ref{sec:MLE_existence}. For further examples see~\cite{RCON, Uhler_2012}.

More general Gaussian models with linear equality constraints on the concentration matrix or the covariance matrix were introduced by Anderson~\cite{andersonLinearCovariance}. He was motivated by the linear structure of covariance and concentration matrices resulting from various time series models. As pointed out in Section~\ref{sec:likelihood}, the Gaussian likelihood as a function of $\Sigma$ is not concave over the whole cone of positive definite matrices. Hence maximum likelihood estimation for Gaussian models with linear constraints on the covariance matrix in general does not lead to a convex optimization problem and has many local maxima. Anderson proposed iterative procedures for calculating the MLE for such models, such as the Newton-Raphson method~\cite{andersonLinearCovariance} and a scoring method~\cite{anderson73}.

As mentioned in Section~\ref{sec:likelihood}, while not being concave over the whole cone of positive definite matrices, the Gaussian likelihood as a function of $\Sigma$ is concave over a large region of $\mathbb{S}^p_{\succ 0}$, namely for all $\Sigma$ that satisfy $\Sigma - 2S\in\mathbb{S}^p_{\succ 0}$. This is useful, since it was shown in~\cite{linear_Gaussian} that the MLE for Gaussian models with linear equality constraints on the covariance matrix lies in this region with high probability as long as the sample size is sufficiently large ($n \simeq 14 p$). Hence in this regime, maximum likelihood estimation for linear Gaussian covariance models behaves as if it were a convex optimization problem.

Similarly as we posed the question for Gaussian graphical models in Section~\ref{sec:MLE_alg}, one can ask when the MLE of a linear Gaussian covariance model has a closed form representation. Szatrowski showed in~\cite{Szatrowski1980, Szatrowski1985} that the MLE for linear Gaussian covariance models has an explicit representation if and only if $\Sigma$ and $\Sigma^{-1}$ satisfy the same linear constraints. This is equivalent to requiring that the linear subspace $\mathcal{L}$, which defines the model, forms a Jordan algebra, i.e., if $\Sigma\in\mathcal{L}$ then also $\Sigma^2\in\mathcal{L}$~\cite{jensen1988}. Furthermore, Szatrowski proved that for this model class Anderson's scoring method~\cite{anderson73} yields the MLE in one iteration when initiated at any positive definite matrix in the model. 

Linear inequality constraints on the concentration matrix also lead to a convex optimization problem for ML estimation. An example of such models are Gaussian distributions that are \emph{multivariate totally positive of order two} (MTP$_2$). This is a form of positive dependence, which for Gaussian distributions implies that $K_{i,j}\leq 0$ for all $i\neq j$. Gaussian MTP$_2$ distributions were studied by Karlin and Rinott~\cite{karlinGaussian} and more recently in~\cite{MTP2_Gaussian, slawski2015estimation} from a machine learning and more applied perspective. It was shown in~\cite{Fallat_2016} that MTP$_2$ distributions have remarkable properties with respect to conditional independence constraints. In addition, for such models the spanning forest of the sample correlation matrix is always a subgraph of the maximum likelihood graph, which can be used to speed up graph learning algorithms~\cite{MTP2_Gaussian}. Furthermore, the MLE for MTP$_2$ Gaussian models exists already for 2 observations with probability~1~\cite{slawski2015estimation}. These properties make MTP$_2$ Gaussian models interesting for the estimation of high-dimensional graphical models.

We end by referring to Pourahmadi~\cite{Pourahmadi} for a comprehensive review of covariance estimation in general and a discussion of numerous other specific covariance matrix constraints.

\bibliographystyle{plain}
\bibliography{biblio}

\bigskip \bigskip

\noindent
{\bf Acknowledgements.} Caroline Uhler was partially supported by DARPA (W911NF-16-1-0551), NSF (DMS-1651995) and ONR (N00014-17-1-2147).

\bigskip
\bigskip

\noindent
{\bf Authors' addresses:}

\smallskip

\noindent Caroline Uhler,
Laboratory for Information and Decision Systems,
Department of Electrical Engineering and Computer Science,
Institute for Data, Systems and Society,
Massachusetts Institute of Technology,
            {\tt cuhler@mit.edu}.

\end{document}